\documentclass[12pt]{amsart}
\usepackage[top=0.9in, bottom=1in, left=1in, right=1in]{geometry}
\usepackage{amsmath, amssymb, amsthm}
\usepackage{verbatim}
\usepackage{graphicx}
\usepackage{enumerate}
\usepackage{mathrsfs}

\newtheorem{thm}{Theorem}[section]

\newtheorem{lem}[thm]{Lemma}
\newtheorem{cor}[thm]{Corollary}

\def\XXint#1#2#3{{\setbox0=\hbox{$#1{#2#3}{\int}$ }
\vcenter{\hbox{$#2#3$ }}\kern-.6\wd0}}

\theoremstyle{definition}
\newtheorem{definition}[thm]{Definition}
\newtheorem{example}[thm]{Example}

\theoremstyle{remark}

\newtheorem{remark}[thm]{Remark}

\numberwithin{equation}{section}

\newtheoremstyle{ser}
{8pt}
{8pt}
{\it}
{}
{\sf}
{:}
{6mm}
{}

\newtheoremstyle{serr}
{8pt}
{8pt}
{\normalfont}
{}
{\sf}
{.}
{6mm}
{}

\theoremstyle{ser}

\theoremstyle{serr}

\begin{document}


\title[Determinant preserving maps]{On a class of determinant preserving maps for finite von Neumann algebras}


\author{Marcell Ga\'al}
\address{Bolyai Institute, Functional Analysis Research Group, University of Szeged, H-6720 Szeged, Aradi v\'ertan\'uk tere 1, Hungary}
\email{marcell.gaal.91@gmail.com}

\author{Soumyashant Nayak}
\address{Smilow Center for Translational Research, University of Pennsylvania, Philadelphia, PA 19104}
\email{nsoum@upenn.edu}
\urladdr{https://nsoum.github.io/} 


\begin{abstract}
Let $\mathscr{R}$ be a finite von Neumann algebra with a faithful tracial state $\tau $ and let $\Delta$ denote the associated Fuglede-Kadison determinant. In this paper, we characterize all unital bijective maps $\phi$ on the set of invertible positive elements in $\mathscr{R}$ which satisfy $$\Delta(\phi(A)+\phi(B)) = \Delta(A+B).$$
We show that any such map originates from a $\tau$-preserving Jordan $*$-automorphism of $\mathscr{R}$ (either $*$-automorphism or $*$-anti-automorphism in the more restrictive case of finite factors). In establishing the aforementioned result, we make crucial use of the solutions to the equation $\Delta(A + B) = \Delta(A) + \Delta(B)$ in the set of  invertible positive  operators in $\mathscr{R}$. To this end, we give a new proof of the inequality $$\Delta(A+B) \ge \Delta(A) + \Delta(B),$$ using a generalized version of the Hadamard determinant inequality and conclude that equality holds for invertible $B$ if and only if $A$ is a nonnegative scalar multiple of $B$.

\end{abstract}

\maketitle

\section{Introduction}
In 1897 Frobenius \cite{Frobenius} proved that if $\phi$ is a linear map on the matrix algebra $M_n(\mathbb{C})$ of $n \times n$ complex matrices preserving the determinant, then there are matrices $M,N\in M_n(\mathbb{C})$ such that $\det(MN)=1$ and $\phi$ can be written in one of the following forms:
\begin{itemize}
\item[a)] $\phi(A)=MAN, \quad A\in M_n(\mathbb{C})$;
\item[b)] $\phi(A)=MA^{t}N, \quad A\in M_n(\mathbb{C})$
\end{itemize}
where ${(\cdot)}^{t}$ denotes transposition of a matrix.

In the past decades this result of Frobenius has inspired many researchers to deal with different sorts of preserver problems involving various notions of determinant \cite{aupetit, dolinar, eaton, szokol, nagy, tan}. Among others, in \cite{szokol} Huang et al. completely described all maps on the positive definite cone $\mathbb{P}_n$ of $M_n(\mathbb{C})$ which satisfy the sole property $\det(\phi(A)+\phi(B))=\det\phi(I)\cdot \det(A+B)$ for all $A,B \in \mathbb{P}_n$. Note that when $\phi$ is a unital linear map on $\mathbb{P}_n$, the above property simply means that $\phi$ is $\det$ preserving. In this paper, we consider an identical operator algebraic counterpart of this problem in the setting of finite von Neumann algebras.

Our approach to the solution is based on a generalization of the Minkowski determinant inequality to the setting of von Neumann algebras. Note that the usual Minkowski determinant inequality for matrices $A, B \in \mathbb{P}_n$ asserts that
 $$\sqrt[n]{\mathrm{det}(A+B)} \ge \sqrt[n]{\mathrm{det}(A)} + \sqrt[n]{\mathrm{det}(B)},$$ with equality if and only if $A, B$ are positive scalar multiples of each other. In \cite[Corollary 4.3.3 (i)]{general-hadamard}, Arveson gives a variational proof of a version of the Minkowski determinant inequality in finite von Neumann algebras involving the Fuglede-Kadison determinant. Recently, this result has been subsumed by the study of the anti-norm property of a wide class of functionals by Bourin and Hiai \cite[Corollary 7.6]{bourin-hiai}. The equality conditions are harder to isolate from these proofs because of limiting arguments and are not explicitly documented. As that will play an important role in our results, we first need to establish when $\Delta(A+B) = \Delta(A) + \Delta(B)$ holds for positive operators $A, B$ in a finite von Neumann algebra. To this end, we give a new proof of the inequality using a generalized version of the Hadamard determinant inequality \cite{hadamard, fischer} in a bootstrapping argument. Consequently, we are able to solve the aforementioned preserver problem concerning Fuglede-Kadison determinants on finite von Neumann algebras.

The paper is organized as follows.
In \S{}~\ref{S:2} we fix the notation and briefly review some facts from the theory of determinants on von Neumann algebras.
The precise formulations of our corresponding results and their proofs are collected in \S{}~\ref{S:3} and \S{}~\ref{S:4}.

\section{Preliminaries} \label{S:2}

Throughout this paper, $\mathscr{R}$ denotes a finite von Neumann algebra acting on the complex (separable) Hilbert space $\mathscr{H}$ 
and containing the identity operator $I$. Let $\tau$ be a faithful tracial state on $\mathscr{R}$, by which we mean a linear functional 
$\tau:\mathcal{\mathscr{R}}\to \mathbb{C}$ such that for all $A,B\in \mathcal{\mathscr{R}}$, we have (i) $\tau(AB)=\tau(BA)$, (ii) 
$\tau(A^*A)\geq 0$ with equality if and only if $A=0$, (iii) $\tau(I)=1$. The set of invertible operators in $\mathscr{R}$ is denoted 
by $GL_1(\mathscr{R})$. We denote the cone of positive operators in $\mathscr{R}$ by $\mathscr{R}^{+}$ and use $GL_1(\mathscr{R})^{+}$ 
to denote the set of invertible operators in $\mathscr{R}^{+}$.

For $A \in GL_1(\mathscr{R})$, the Fuglede-Kadison determinant $\Delta$ associated with $\tau$ is defined as $$\Delta(A) = \exp(\tau(\log \sqrt{A^*A})).$$ 
The dependence of $\Delta$ on $\tau $ is suppressed in the notation and it is to be assumed that a choice of a faithful tracial state has already been made. Although this concept of determinant was developed in \cite{fuglede-kadison} in the context of type II$_{1}$ factors, it naturally extends to finite von Neumann algebras as above.

\begin{example}
\label{rmrk:fuglede-kadison}
The simplest examples of finite von Neumann algebras are given by $M_n(\mathbb{C})$, the full matrix algebra of $n \times n$ complex matrices. For $A \in M_n(\mathbb{C})$, the Fuglede-Kadison determinant $\Delta(A)$ is given by $\sqrt[n]{|\mathrm{det}(A)|}$ where $\mathrm{det}$ is the usual matrix determinant.
\end{example} 

\begin{example}
On $M_2(\mathbb{C})$ the unique faithful tracial state is given by $$\mathrm{tr}_2 : M_2(\mathbb{C}) \rightarrow \mathbb{C},\quad \mathrm{tr}_2(A) = \frac{a_{11} + a_{22}}{2}$$ where $a_{ij} \in \mathbb{C}$ $(1 \le i, j \le 2)$ denotes the $(i, j)^{\mathrm{th}}$ entry of the matrix $A$ in $M_2(\mathbb{C})$. Denote by $D_2(\mathbb{C})$ the $*$-subalgebra of diagonal matrices in $M_2(\mathbb{C})$. The von Neumann algebra $M_2(\mathscr{R}) \cong \mathscr{R} \otimes M_2(\mathbb{C})$ (acting on $\mathscr{H} \oplus \mathscr{H}$) is also finite and the faithful tracial state on $M_2(\mathscr{R})$ is given by $\tau_2 = \tau \otimes \mathrm{tr}_2$, that is, for an operator $A$ in $M_2(\mathscr{R})$, we have $$\tau_2(A) = \frac{\tau(A_{11}) + \tau(A_{22})}{2}.$$ 
We denote by $\Delta_2$ the Fuglede-Kadison determinant on $M_2(\mathscr{R})$ corresponding to $\tau \otimes \mathrm{tr}_2$. For operators $A_1, A_2$ in $\mathscr{R}$, we define 
$$\mathrm{diag}(A_1, A_2) := \begin{bmatrix}
A_1 & 0\\
0 & A_2
\end{bmatrix} \in D_2(\mathscr{R}) \cong \mathscr{R} \otimes D_2(\mathbb{C}) \subset M_2(\mathscr{R}).$$ It is straightforward to see that for invertible operators $A_1, A_2$ in $\mathscr{R}$, the operator $\mathrm{diag}(A_1, A_2)$ in $M_2(\mathscr{R})$ is invertible and $\Delta_2(\mathrm{diag}(A_1, A_2)) = \sqrt{\Delta(A_1) \cdot \Delta(A_2)}$.
\end{example}

\begin{example}
Let $X$ be a compact (Hausdorff) topological space with a probability Radon measure $\nu$. The space of essentially bounded complex-valued functions on $(X, \nu)$, denoted by $L^{\infty}(X,\nu)$, which acts via left multiplication on $L^2(X, \nu)$, forms an abelian von Neumann algebra. The involution operation is given by $f^*(x):=\overline{f(x)}$. A faithful tracial state on $L^{\infty}(X, \nu)$ is obtained by
\[
\tau_{\nu}(f)=\int_{G} f(x) \; d\nu(x), \textrm{ for }f\in L^{\infty}(X,\nu)
\]
and the corresponding Fuglede-Kadison determinant determinant is given by
\[
\Delta_{\nu}(f)=\exp\left(\int_{G} \log(\left|f(x)\right|) d\nu(x)\right).
\]
\end{example}

Group von Neumann algebras provide another important class of examples of finite von Neumann algebras, see e.g. \cite[\S 3.2]{sinclair-smith}.

One of the most remarkable properties of $\Delta$ is that it is a group homomorphism of $GL_1(\mathscr{R})$ into the multiplicative group of positive real numbers. However, there may be several extensions of $\Delta$ from $GL_1(\mathscr{R})$ to the whole of $\mathscr{R}$. From the proof of \cite[Lemma 6]{fuglede-kadison}, note that for a projection $E \ne I$ in $\mathscr{R}$, we have $\Delta '(E) = 0$ for any extension $\Delta '$ of $\Delta$. In this paper, we consider only the {\it analytic extension} which is defined as follows. For $A \in \mathscr{R}$, let $\sigma(|A|) \subset [0, \infty)$ denote the spectrum of $\sqrt{A^*A}$ and let $\mu$ be the probability measure supported on $\sigma(|A|)$ and induced by the tracial state $\tau$. Then we define $$\Delta(A):=\exp\left(\int_{\sigma(|A|)} \log \lambda \; d\mu(\lambda)\right)$$ with understanding that $\Delta(A) = 0$ whenever $\int_{\sigma(|A|)} \log \lambda \; d\mu(\lambda) = -\infty$. We abuse notation and denote this extension also by $\Delta$. 

Below we summarize some properties of $\Delta$ which we shall need in \S{}~\ref{S:3}.
\begin{itemize}
\item[(p1)] $\Delta(U) = 1$ for a unitary $U$ in $\mathscr{R}$;
\item[(p2)] $\Delta(AB) = \Delta(A)\cdot \Delta(B)$ for $A, B \in \mathscr{R}$;
\item[(p3)] $\Delta$ is norm continuous on $GL_1(\mathscr{R})$;
\item[(p4)] $\Delta(\lambda A) = |\lambda| \Delta(A)$ for $\lambda \in \mathbb{C}, A \in \mathscr{R}$;
\item[(p5)] $\lim_{\varepsilon \rightarrow 0{+}} \Delta(A + \varepsilon I) = \Delta(A)$ for a positive operator $A$ in $\mathscr{R}$.
\end{itemize}

\section{The Minkowski determinant inequality} \label{S:3}

In this section, we aim to establish the following version of the Minkowski determinant inequality.

\begin{thm}[generalized Minkowski determinant inequality]
\label{thm:minkowski-final}
{\sl For positive operators $A, B$ in $\mathscr{R}$, we have
\begin{equation}
\label{eqn:minkowski-final}\Delta(A+B) \ge \Delta(A) + \Delta(B).
\end{equation} 
Moreover, if $B$ is invertible, then equality holds in \eqref{eqn:minkowski-final} if and only if $A$ is a nonnegative scalar multiple of $B$.}
\end{thm}
We work towards the proof of Theorem~\ref{thm:minkowski-final} using several lemmas. Before turning to their proof, let us explain the main ideas implemented in them.

A proof of the Minkowski determinant inequality (see \cite[p. 115]{marcus-minc}) for matrices is based on the 'traditional' Hadamard determinant inequality which states that for a positive definite matrix $A$ in $M_n(\mathbb{C})$, the determinant of $A$ is less than or equal to the product of its diagonal entries and equality holds if and only if $A$ is a diagonal matrix. For a given $A\in M_n(\mathbb{C})$, considering the positive semidefinite matrix $\sqrt{A^*A}$, one may derive from this inequality the geometrically intuitive fact that the volume of an $n$-parallelepiped with prescribed lengths of edges is maximized when the edges are mutually orthogonal. In this paper, we make use of an 'abstract' Hadamard-type determinant inequality in our proof to reflect the geometric origins of inequality (\ref{eqn:minkowski-final}).

Recall that if $\mathscr{S}$ is a von Neumann subalgebra of $\mathscr{R}$, then by a \emph{conditional expectation} we mean a unital (identity preserving) positive linear map $\Phi: \mathscr{R}\to \mathscr{S}$ which satisfies $\Phi(SAT)=S\Phi(A)T$ for all $A\in \mathscr{R}$ and $S,T\in \mathscr{S}$. Concerning $\tau$-preserving conditional expectations, in \cite[Theorem 4.1]{nayak-hadamard} the second author has proved the following generalization of the Hadamard determinant inequality:
\begin{thm}
\label{thm:hadamard}
{\sl
For a $\tau$-preserving conditional expectation $\Phi$ on $\mathscr{R}$ and an invertible positive operator $A$ in $\mathscr{R}$, we have that 
$$\Delta(\Phi(A^{-1})^{-1}) \le \Delta(A) \le \Delta(\Phi(A))$$ and equality holds in either of the above two inequalities (and hence in both inequalities) if and only if $\Phi(A) = A$.}
\end{thm}
We consider the map $\Phi_2 : M_2(\mathscr{R}) \rightarrow M_2(\mathscr{R})$ defined by $$\Phi_2(A) := \mathrm{diag}(A_{11}, A_{22}) \in D_2(\mathscr{R}) \subset M_2(\mathscr{R})$$ where $A_{ij} \in \mathscr{R} (1 \le i, j \le 2)$ is the $(i, j)^{\mathrm{th}}$ entry of $A$. Note that $\Phi_2$ is a ($\tau \otimes \mathrm{tr}_2$)-preserving normal conditional expectation from $M_2(\mathscr{R})$ onto the von Neumann subalgebra $D_2(\mathscr{R})$. In Lemma \ref{lem:minkowski-fischer}, we will use Theorem \ref{thm:hadamard} in the context of the von Neumann algebra $M_2(\mathscr{R})$ and the ($\tau \otimes \mathrm{tr}_2$)-preserving normal conditional expectation $\Phi_2$. More precisely, for positive operators $A_1, A_2$ in $\mathscr{R}$ we first prove in Lemma \ref{lem:minkowski-fischer} that 
\begin{equation}
\label{ineq:prelim-minkowski}
\Delta(A_1) \cdot \Delta(A_2) \le \Delta(tA_1 + (1-t)A_2) \cdot \Delta(tA_2 + (1-t)A_1), \quad t \in [0, 1].
\end{equation}
Choosing $t = 1/2$, we arrive at
\begin{align*}
&\Delta(A_1) \cdot \Delta(A_2) \le \Delta\left(\frac{A_1 + A_2}{2}\right)^2
\end{align*}
which readily implies that $2\sqrt{\Delta(A_1) \cdot \Delta(A_2)} \le \Delta(A_1 + A_2)$, which is weaker than the desired inequality. We then use a "tensor power trick" to proceed with a bootstrapping argument to prove the required inequality.

Now we are in a position to prove our first lemma.

\begin{lem}
\label{lem:minkowski-fischer}
{\sl For positive operators $A_1, A_2$ in $\mathscr{R}$ and $t \in [0,1]$, the following inequality holds: 
\begin{equation}
\label{ineq:minkowski-fischer}
\Delta(A_1) \cdot \Delta(A_2) \le \Delta(tA_1 + (1-t)A_2) \cdot \Delta(tA_2 + (1-t)A_1).
\end{equation}
Further if $A_1, A_2$ are invertible, then equality holds if and only if either $t \in \{0, 1\}$ or $A_1 = A_2$. }
\end{lem}

\begin{proof}
Consider the unitary operator $U$ in $M_2(\mathscr{R})$ given by
$$U :=  \begin{bmatrix}
\sqrt{t} I & \sqrt{1-t} I \\
\sqrt{1-t} I & -\sqrt{t} I
\end{bmatrix}.$$
Note that $$U^* \mathrm{diag}(A_1, A_2) U = \begin{bmatrix}
tA_1 + (1-t)A_2 & \sqrt{t(1-t)}(A_1 - A_2)\\
\sqrt{t(1-t)}(A_1 - A_2) & tA_2 + (1-t)A_1
\end{bmatrix}.$$
Clearly, $\Phi_2(U^*\mathrm{diag}(A_1, A_2)U) = \mathrm{diag}(tA_1 + (1-t)A_2, tA_2 + (1-t)A_1)$. Using Theorem~\ref{thm:hadamard} and property (p5) concerning $\Delta$, we get that 
\begin{align*}
\sqrt{\Delta(A_1)} \cdot \sqrt{\Delta(A_2)} = \Delta_{2}(\mathrm{diag}(A_1, A_2 )) = \Delta_2 (U^* \mathrm{diag}(A_1,A_2) U) \\
\le \Delta_2(\Phi_2(U^* \mathrm{diag}(A_1, A_2)U)) = \sqrt{\Delta(tA_1 + (1-t)A_2)} \cdot \sqrt{\Delta(tA_2 + (1-t)A_1)}.
\end{align*}
If $A_1, A_2$ are invertible, then $U^* \mathrm{diag}(A_1, A_2)U$ is also invertible and equality holds if and only if $\sqrt{t(1-t)}(A_1 - A_2) = 0$, that is, either $t \in \{0, 1\}$ or $A_1 = A_2$.
\end{proof}

\begin{lem}
\label{lem:minkowski-prep}
{\sl Let $n$ be a positive integer. For positive operators $A_1, A_2, \ldots, A_n$ in $\mathscr{R}$, the following inequality holds: 
\begin{equation}
\label{eqn:minkowski-prep}
\Delta\Big(\frac{A_1 + \cdots + A_n}{n} \Big) \ge (\Delta(A_1) \cdots \Delta(A_n))^{1/n}.
\end{equation} Further if $A_1, \ldots, A_n$ are invertible, then equality holds if and only if $A_1 = A_2 = \ldots = A_n$.}
\end{lem}
\begin{proof}
First for $n = 2^k$ with $k \in \mathbb{N}$ we prove by induction the inequality \eqref{eqn:minkowski-prep} along with the equality condition and then employ a standard argument to establish it for all $n \in \mathbb{N}$. Choosing $t = 1/2$ in Lemma \ref{lem:minkowski-fischer}, we get for $A_1, A_2 \in \mathscr{R}^{+}$ that 
\begin{equation} \label{basecase}
\Delta \Big( \frac{A_1 + A_2}{2} \Big) \ge (\Delta(A_1) \cdot \Delta(A_2))^{1/2}.
\end{equation}
Further if $A_1, A_2$ are invertible, equality holds if and only if $A_1 = A_2.$ This proves the case when $n=2$.

Now assume that inequality \eqref{eqn:minkowski-prep} holds for $n = 2^{k-1}$ along with the equality condition. For $A_1, A_2, \ldots, A_{2^k}\in \mathscr{R}^+$, we define $$B_1 := \frac{A_1 + \cdots + A_{2^{k-1}}}{2^{k-1}}, \quad B_2 := \frac{A_{2^{k-1}+1} + \cdots + A_{2^k}}{2^{k-1}}.$$
From \eqref{basecase}, we infer that 
\begin{equation}
\label{eqn:base-case}
\Delta \Big(\frac{B_1 + B_2}{2} \Big) \ge (\Delta(B_1) \Delta(B_2))^{1/2}.
\end{equation}
Furthermore, the induction hypothesis furnishes
\begin{equation}
\label{eqn:induct-hyp}
\Delta(B_1) \ge (\Delta(A_1) \cdots \Delta(A_{2^{k-1}}) )^{1/2^{k-1}}, \quad\Delta(B_2) \ge (\Delta(A_{2^{k-1} + 1}) \cdots \Delta(A_{2^k}))^{1/2^{k-1}}.
\end{equation}
Combining \eqref{eqn:base-case} and \eqref{eqn:induct-hyp}, we have
\begin{equation}
\Delta \Big(\frac{A_1 + \cdots + A_{2^k}}{2^k} \Big) \ge (\Delta(A_1) \cdots \Delta(A_{2^k}))^{1/2^k}.
\end{equation}
If $A_1, \ldots, A_{2^k}$ are invertible, then so are $B_1, B_2$, and equality holds if and only if $B_1 = B_2$ $\quad A_1 = \cdots = A_{2^{k-1}}$ and $A_{2^{k-1}+1} = \cdots = A_{2^k}$ or, in other words, if and only if $A_1 = \cdots  = A_{2^k}$. Thus by induction, for $n$ a power of $2$, we have established inequality \eqref{eqn:minkowski-prep} along with the equality condition.

Next we consider an arbitrary positive integer $m$. Let $k$ be a positive integer such that $2^{k-1} \le m < 2^k$. For positive operators $A_1, \ldots, A_m,$ we define a positive operator $B$ by $B := (A_1 + \cdots + A_m)/m$.
It follows that
\begin{align*}
\Delta(B) = \Delta \Big(\frac{A_1 + \cdots + A_m + (2^k - m)B}{2^k} \Big) \ge (\Delta(A_1) \cdots \Delta(A_m))^{1/2^k} (\Delta(B))^{1 - m/2^k}
\end{align*}
and using property (p5), we conclude that
\begin{align*}
\Delta(B) ^{m/2^k} \ge (\Delta(A_1) \cdots \Delta(A_m))^{1/2^k}
~ \Longrightarrow ~ \Delta(B) \ge (\Delta(A_1) \cdots \Delta(A_m))^{1/m}.
\end{align*}
If $A_1, \ldots, A_m$ are invertible, then so is $B$ and equality holds if and only if $A_1 = A_2 = \ldots = A_m = B$.
\end{proof}

\begin{thm}
\label{thm:minkowski}
{\sl For a positive operator $A$ in $\mathscr{R}$ and $t \ge 0$, the following inequality holds:
\begin{equation}
\label{eqn:minkowski}
\Delta(tI+A) \ge t + \Delta(A)
\end{equation} with equality if and only if either $t = 0$ or $A$ is a nonnegative scalar multiple of $I$.}
\end{thm}

\begin{proof}
Let $A$ be an invertible positive operator such that $\Delta(A) = 1$. For $p, q \in \mathbb{N}$, an application of Lemma~\ref{lem:minkowski-prep} gives us that $$\Delta\left(\frac{pI + qA}{p+q}\right) \ge \sqrt[p+q]{\Delta(I)^p \Delta(A)^q} = 1.$$ Thus $\Delta((p/q)I + A) \ge p/q + 1$ with equality if and only if $A = I$. Approximating with strictly positive rational numbers, we have by property (p5) for $\Delta$ that
\begin{equation}
\label{eqn:ineq}
\Delta(tI + A) \ge t + 1, \quad \textrm{ for } t \ge 0.
\end{equation} 
Note that for $A\in GL_1(\mathscr{R})^{+}$ the operator $B := (1/\Delta(A))A$ is an invertible positive operator satisfying $\Delta(B) = 1$. As $\Delta(tI + B) \ge t + 1$, for $t \ge 0$ substituting $s = t \Delta(A)$, we get the desired inequality
\begin{equation}
\label{ineq:inside}
\Delta(sI+A) \ge s + \Delta(A), \quad \textrm{ for } s \ge 0.
\end{equation} 

Next we derive conditions for the case of equality in \eqref{ineq:inside}. Note that for a particular value of $s$ under consideration, if $s/\Delta(A)$ is rational, then equality holds in \eqref{ineq:inside} if and only if $B = I \Leftrightarrow A = \Delta(A) I$. For $s > 0$, if $\Delta(sI + A) = s + \Delta(A)$, using \eqref{ineq:inside} repeatedly along with the multiplicativity of $\Delta$, we get that 
\begin{align*}
s^2 + \Delta(A)(2s + \Delta(A)) &= (s + \Delta(A))^2 = \Delta (sI + A)^2 = \Delta(s^2I + A(2sI + A)) \\
&\ge s^2 + \Delta(A) \Delta(2sI + A) \ge s^2 + \Delta(A)(2s + \Delta(A)).
\end{align*}
Thus we conclude that $\Delta(2sI + A) = 2s + \Delta(A)$. For $r \in ]0, s[$, using \eqref{ineq:inside} we deduce that
\begin{align*}
2s + \Delta(A) = \Delta(2sI + A) &= \Delta((s+r)I + (s-r)I + A)\\
& \ge (s+r) + \Delta((s-r)I + A) \\
& \ge (s+r) + (s-r) + \Delta(A) = 2s + \Delta(A).
\end{align*}
As $(s+r) + \Delta((s-r)I + A) = 2s + \Delta(A)$, we have that $\Delta((s-r)I + A) = s-r + \Delta(A)$ for all $r \in ]0, s[$. We may choose $r$ such that $(s-r)/\Delta(A)$ is rational and thus conclude that $A$ is a scalar multiple of the identity. Hence equality holds in \eqref{ineq:inside} if and only if either $s = 0$ or $A$ is a scalar multiple of the identity.

Next we consider the case when $A$ is not necessarily invertible. For some $t > 0$ and any $s \in ]0, t]$, define $A_{s} := sI + A$. As documented in property (p5) for $\Delta$, note that $\lim_{\varepsilon \rightarrow 0+} \Delta(A_{\varepsilon}) = \Delta(A)$. We have $$\Delta(A_t) = \Delta\left(\frac{t}{2}I + A_{t/2}\right) \ge \frac{t}{2} + \Delta(A_{t/2}) \ge \sum_{i=1}^k \frac{t}{2^i} + \Delta(A_{t / 2^k}).$$
Taking the limit $k \rightarrow \infty$, we conclude that $$\Delta(tI + A) = \Delta(A_t) \ge \sum_{i=1}^{\infty} \frac{t}{2^i} + \Delta(A) = t + \Delta(A). $$ If $A$ is a scalar multiple of the identity, equality trivially holds. If $\Delta(tI + A) = t + \Delta(A)$, we must have $\Delta(A_t) = t/2 + \Delta(A_{t/2})$ and thus $A_{t/2}$ is a scalar multiple of $I$ implying that $A$ is a scalar multiple of $I$.
\end{proof}

\begin{proof}[Proof of Theorem~\ref{thm:minkowski-final}]
If $B$ is invertible, then by \eqref{eqn:minkowski} we infer that $$\Delta(I + B^{-\frac{1}{2}}AB^{-\frac{1}{2}}) \ge 1 + \Delta(B^{-\frac{1}{2}}AB^{-\frac{1}{2}})$$ with equality if and only if $B^{-\frac{1}{2}}AB^{-\frac{1}{2}} = \lambda I$ with some $\lambda \ge 0$. Using the multiplicative property of the determinant $\Delta$, we conclude that $\Delta(A + B) \ge \Delta(A) + \Delta(B)$ with equality if and only if $A = \lambda B$ with some $\lambda \ge 0$.

If $B$ is not invertible, we consider the invertible positive operator $B_{\varepsilon} := B + \varepsilon I$ $(\varepsilon > 0)$. Then we have $\Delta(A + B_{\varepsilon}) \ge \Delta(A) + \Delta(B_{\varepsilon})$ and taking the limit $\varepsilon \rightarrow 0+$, we see that $\Delta(A+B) \ge \Delta(A) + \Delta(B)$, as required.
\end{proof}

\begin{remark}
If $A$ or $B$ is invertible, then the condition of equality has a straightforward form demanding a scaling relationship between the operators unless one of them is $0$. But if neither $A$ nor $B$ is invertible, then the conditions under which equality holds are less easy to characterize. Assume that $E, F$ are orthogonal projections in $\mathscr{R}$ such that $E + F < I$. We then have $\Delta( E + F) = 0 = \Delta (E) + \Delta(F)$ and thus we cannot expect the operators in question to have any specific 'correlation', a term which we do not define but whose spirit is captured in the above example.
\end{remark}

\section{A class of determinant preserving maps} \label{S:4}

In this section, we present our result concerning certain determinant preserving bijective maps on $GL_1(\mathscr{R})^{+}$.
Before we do so let us recall the concept of Jordan $*$-isomorphisms.

\begin{definition}
A linear map $J:\mathscr{R}\to \mathscr{R}$ is called a
\begin{itemize}
\item[(i)] Jordan homomorphism if $J(A^2)=J(A)^2$, for all $A\in \mathscr{R}$;
\item[(ii)] Jordan $*$-homomorphism if $J(A^*) = J(A)^*$ for all $A \in \mathscr{R}$ and $J$ is a Jordan homomorphism;
\item[(ii)] Jordan $*$-isomorphism if $J$ is a bijective Jordan $*$-homomorphism.
\end{itemize}
\end{definition}

Whenever we use the terms $*$-homomorphism, $*$-isomorphisms and $*$-automorphism, it is implicitly understood to refer to the $C^*$-algebraic structure of $\mathscr{R}$. A celebrated result of Kadison \cite[Corollary 5]{kadison-schwarz} states that a (linear) order automorphism of a $C^*$-algebra is necessarily implemented by a Jordan $*$-automorphism. 
Roughly speaking, this means that in a $C^*$-algebra the order and the Jordan structures are intimately connected and in fact, determine each other.
The mentioned result of Kadison was crucially used in \cite[Lemma 8]{kloop} where the structure of additive bijective maps was determined on the cone $GL_1(\mathscr{R})^{+}$. We apply this in the proof of the main result of the section which is as follows.

\begin{thm}\label{T:preserver}
\textsl{
Let $\phi: GL_1(\mathscr{R})^{+}\to GL_1(\mathscr{R})^{+}$ be a bijection. Then
\[
\Delta(\phi(A)+\phi(B))=\Delta(\phi(I))\cdot \Delta(A+B)
\]
holds for all $A,B$ in $GL_1(\mathscr{R})^{+}$ if and only if there is a $\tau$-preserving Jordan $*$-isomorphism $J:\mathscr{R}\to \mathscr{R}$ and a positive invertible element $T\in GL_1(\mathscr{R})^{+}$ such that
\[
\phi(A)=TJ(A)T, \quad A\in GL_1(\mathscr{R})^{+}.
\]
}
\end{thm}

The proof uses the main ideas in \cite{szokol}, however, we must adapt them to the much more general setting of finite von Neumann algebras. 
Before turning to the proof of Theorem \ref{T:preserver}, we paraphrase an auxiliary lemma from \cite{fuglede-kadison} which makes use of the Riesz-Dunford holomorphic functional calculus for Banach algebras \cite{dunford} to derive a pertinent corollary.

\begin{lem}[{\cite[Lemma 2]{fuglede-kadison}}]
\label{lem:tech-preserver}
\textsl{
Let $\mathscr{B}$ be a complex Banach algebra with a norm-continuous tracial linear functional $\mathfrak{T}$. Let $f(\lambda)$ be a holomorphic function on a domain $\Lambda \subset \mathbb{C}$ bounded by a curve $\Gamma$ and let $\gamma : [0,1] \rightarrow \mathscr{R}$ be a differentiable family of operators in $\mathscr{R}$, such that the spectrum of each operator $\gamma(t)$ lies in $\Lambda$. Then $f(\gamma(t))$ is differentiable with respect to $t$ and $$ \mathfrak{T}( (f \circ \gamma)'(t) )) = \mathfrak{T}(f'(\gamma(t)) \cdot  \gamma '(t)).$$
}
\end{lem}

\begin{cor}
\label{cor:tech-preserver}
\textsl{
For invertible positive operators $A, B$ in $\mathscr{R}$, the function $g : [0, 1] \rightarrow \mathbb{R}$ defined by $g(t) = \Delta(tA + (1-t)B)$ is differentiable at $0+$ and $g'(0+) = \Delta(B) (\tau(B^{-1}A) - 1).$}
\end{cor}

\begin{proof}
Let $\gamma : [0,1] \rightarrow GL_1(\mathscr{R})^{+}$ be the line segment $\gamma(t) = tA + (1-t)B$. Clearly, $\gamma$ is a continuously differentiable curve with $\gamma '(t) = A - B$ for all $t \in [0, 1]$. If $\varepsilon > 0$ is such that $\varepsilon I \le A$ and $\varepsilon I \le B$, we have that $\varepsilon I \le \gamma(t)$ for any $t \in [0,1]$. Thus we may choose a domain $\Lambda \subseteq \mathbb{C}$ not containing $0$ that is bounded by a curve $\Gamma$ which surrounds the spectra of $ \gamma(t)$ for all $t \in [0,1]$ and does not wind around $0$. On the domain $\Lambda$, $f = \log$ is a holomorphic function. Define $G(t): = \tau( \log ( \gamma (t))))$. Using Lemma \ref{lem:tech-preserver}, we get that $$G'(0+) = \tau(B^{-1}(A-B)) = \tau(B^{-1}A - I) = \tau(B^{-1}A) - 1.$$ As $g(t) = \exp G(t)$, we conclude that $g'(0+) = \exp(G(0)) \cdot G'(0+) = \Delta(B) (\tau(B^{-1}A) - 1).$
\end{proof}

\begin{proof}[Proof of Theorem~\ref{T:preserver}]
Let us begin with the necessity part. Observe that the transformation $\psi: GL_1(\mathscr{R})^{+}\to GL_1(\mathscr{R})^{+}$
defined by $\psi(A) = \phi(I)^{-1/2}\phi(A)\phi(I)^{-1/2}$ is unital. From the multiplicativity of $\Delta$, we see that 
\begin{equation}
\label{eqn:additive}
\Delta(\psi(A)+\psi(B))= \Delta(A+B),\quad \textrm{for } A,B \in GL_1(\mathscr{R})^{+}.
\end{equation}
Plugging $A=B$ into equation (\ref{eqn:additive}), we deduce that $\Delta(\psi(A))=\Delta(A)$ for every $A\in GL_1(\mathscr{R})^{+}$ and thus for a positive real number $\lambda > 0$ we obtain
\[
\begin{gathered}
\Delta(\psi(A)+ \psi( \lambda A)) = \Delta(A + \lambda A) =
\Delta(A)+\Delta(\lambda A) = \Delta(\psi(A))+\Delta(\psi(\lambda A)).
\end{gathered}
\]
An application of Theorem~\ref{thm:minkowski-final} entails that $\psi(\lambda A)=\mu \psi(A)$ for some $\mu > 0$. As noted earlier, $\Delta(\psi(A))=\Delta(A)$ which implies $\lambda=\mu$ meaning that $\psi$ is positive homogeneous. For $A,B\in GL_1(\mathscr{R})^{+}$ and a number $t\in ]0,1[$, we get that
\begin{align*}
\Delta(t\psi(A)+(1-t)\psi(B))&= \Delta\left(\psi(tA) + \psi((1-t)B)\right)\\
&=\Delta(tA+(1-t)B).
\end{align*}
For $s \in \{ 0, 1\}$, as $\Delta(\psi(A) ) = \Delta(A)$, $\Delta(\psi(B)) = \Delta(B)$, it follows that $\Delta(s\psi(A)+(1-s)\psi(B) ) = \Delta( sA + (1-s) B )$. In summary, we have for $t \in [0, 1]$ that $\Delta(t\psi(A) + (1-t) \psi(B) ) = \Delta(tA + (1-t)B)$.
Taking the derivative of both sides with respect to $t$ at $0+$, and using Corollary \ref{cor:tech-preserver}, we obtain that
\begin{equation} \label{preserver}
\tau(\psi(B)^{-1}\psi(A)) = \tau(B^{-1}A),\textrm{ for all } A,B\in GL_1(\mathscr{R})^{+}.
\end{equation}
The right hand side of \eqref{preserver} is additive in the variable $A$. As $B$ runs through the whole of $GL_1(\mathscr{R})^{+}$, substituting $X = \psi(B)^{-1}$, it follows from \eqref{preserver} that for all $A, C, X\in GL_1(\mathscr{R})^{+}$, we must have
\[
\tau(X\psi(A+C))=\tau(X \psi(A) ) + \tau(X \psi(C)),
\]
or, equivalently,
\[
\tau(X[\psi(A + C)-(\psi(A)+\psi(C))])=0.
\]
Since a self-adjoint operator $X$ in $\mathscr{R}$ may be written as the difference of two invertible positive operators $X + (\|X\| + \varepsilon)I, (\|X\| + \varepsilon)I $ for $\varepsilon > 0$, we further have that $\tau(X[\psi(A + C)-(\psi(A)+\psi(C))]) = 0$  for all $A, C \in GL_1(\mathscr{R})^{+}$ and all self-adjoint operators $X$ in $\mathscr{R}$. Choosing $X = \psi(A + C)-(\psi(A)+\psi(C))$ and using the faithfulness of the tracial state $\tau$, we conclude that $\psi(A + C)-(\psi(A)+\psi(C)) = 0$ for all $A, C \in GL_1(\mathscr{R})^{+}$. Thus, $\psi$ is an additive bijection.
The structure of such maps is described in \cite{kloop}. According to \cite[Lemma 8]{kloop} there is a Jordan $*$-isomorphism $J:\mathscr{R}\to \mathscr{R}$ such that $\psi(A) = J(A)$ for all $A\in GL_1(\mathscr{R})^{+}$. The desired $\tau$-preserving property also follows from  \eqref{preserver}. Setting $T:=\sqrt{\phi(I)}$ completes the necessity part.

We next prove the sufficiency. It is well-known that for a Jordan $*$-homomorphism $J$ on $\mathscr{R}$ and a continuous function $f$ defined on the spectrum of a self-adjoint operator $A$ in $\mathscr{R}$, we have $J(f(A))=f(J(A)).$ As $J$ is assumed to be $\tau$-preserving, we observe that $\Delta(J(A)) = \Delta(A)$, by taking $f = \log$. With these considerations in mind and by the multiplicativity of $\Delta$ we finally conclude that, if there is a $\tau$-preserving Jordan $*$-homomorphism $J : \mathscr{R} \rightarrow \mathscr{R}$ and an operator $T$ in $GL_1(\mathscr{R})^{+}$ such that $\Phi(A) = TJ(A)T$ for all $A \in GL_1(\mathscr{R})^{+}$, then we must have  $$\Delta(\phi(A + B)) = \Delta(T) \Delta(J(A+B)) \Delta(T) = \Delta(T^2) \Delta(A+B) =  \Delta(\phi(I)) \cdot \Delta(A + B).$$ 
\end{proof}

In the particular case of finite \emph{factors} (that is, von Neumann algebras with trivial center $\mathbb{C}I$), the structure of Jordan $*$-automorphisms is quite straightforward from \cite[Theorem I]{herstein} as finite factors are simple rings (see \cite[Corollary 6.8.4]{kadison-ringrose}). This helps us elucidate the solution to the preserver problem considered in this paper in a simple manner. Note that $*$-automorphisms and $*$-anti-automorphisms of a finite factor are trace preserving.

\begin{cor}
\textsl{
Let $\mathscr{R}$ be a finite factor and $\phi: GL_1(\mathscr{R})^{+}\to GL_1(\mathscr{R})^{+}$ be a bijective map. Then we have
\[
\Delta(\phi(A)+\phi(B))=\Delta(\phi(I))\cdot \Delta(A+B),\textrm{ for all } A,B \in GL_1(\mathscr{R})^{+}
\]
if and only if there is a $*$-automorphism (or $*$-anti-automorphism) $\theta$ of $\mathscr{R}$ and a positive invertible element $T\in GL_1(\mathscr{R})^{+}$ such that
\[
\phi(A)=T\theta(A)T, \quad A\in GL_1(\mathscr{R})^{+}.
\] }
\end{cor}

\section{Acknowledgement}
The first author was supported by the National Research, Development and Innovation Office NKFIH Reg. No. K-115383.

\end{document}